\documentclass{amsart}

\usepackage{a4wide,amsmath,amssymb}
\usepackage[toc,page]{appendix}

\usepackage{url}
\newtheorem{thm}{Theorem}[section]
\newtheorem{lemma}[thm]{Lemma}
\newtheorem{prop}[thm]{Proposition}
\newtheorem{cor}[thm]{Corollary}
\theoremstyle{remark}
\newtheorem{rem}[thm]{Remark}

\newtheorem{ex}[thm]{Example}
\theoremstyle{definition}
\newtheorem{defn}[thm]{Definition}
\newtheoremstyle{Claim}{}{}{\itshape}{}{\itshape\bfseries}{:}{ }{#1}
\theoremstyle{Claim}

\newcommand{\R}{\mathbb{R}}

\newcommand{\He}{\mathbb{H}}

\newcommand{\eps}{\varepsilon}

\newcommand{\al}{\alpha}
\newcommand{\Tr}{\text{Tr}}
\newcommand{\Om}{\Omega}

\DeclareMathOperator{\Prop}{Prop}

\DeclareMathOperator{\sign}{sign}
\DeclareMathOperator{\LSC}{LSC}
\DeclareMathOperator{\USC}{USC}
\newcommand{\Sym}{\mathcal{S}}
\theoremstyle{plain}

\def\sideremark#1{\ifvmode\leavevmode\fi\vadjust{
\vbox to0pt{\hbox to 0pt{\hskip\hsize\hskip1em
\vbox{\hsize3cm\tiny\raggedright\pretolerance10000
\noindent #1\hfill}\hss}\vbox to8pt{\vfil}\vss}}}

\begin{document}

\title[Strong maximum  principles]{New strong maximum and comparison principles\\ for fully nonlinear degenerate elliptic PDEs}

\author{Martino Bardi}
\author{Alessandro Goffi} 

\date{\today}
\subjclass[2010]{Primary: 35B50,35J70,35J60; Secondary: 49L25,35H20.}
\keywords{Fully nonlinear equation, degenerate elliptic equation, subunit vector, H\"ormander condition, strong maximum principle, Hopf boundary lemma, strong comparison principle
.}
 \thanks{
 The authors are members of the Gruppo Nazionale per l'Analisi Matematica, la Probabilit\`a e le loro Applicazioni (GNAMPA) of the Istituto Nazionale di Alta Matematica (INdAM). The first-named author was partially supported by the research projects  ``Mean-Field Games and Nonlinear PDEs'' of the University of Padova, and  ``Nonlinear Partial Differential Equations: Asymptotic Problems and Mean-Field Games" of the Fondazione CaRiPaRo. 
 The second-named author wishes to thank the Department of Mathematics of the University of Padova for the 
  hospitality during the preparation of this paper. }
\address{Department of Mathematics "T. Levi-Civita", University of Padova, Via Trieste 63, 35121 Padova, Italy} \email{bardi@math.unipd.it}
\address{Gran Sasso Science Institute, Viale Francesco Crispi 7, 67100 L'Aquila, Italy } \email{alessandro.goffi@gssi.it}

\maketitle
\begin{abstract}
We introduce a notion of subunit vector field for fully nonlinear degenerate elliptic equations. We prove that an interior maximum of a viscosity subsolution of such an equation propagates along the trajectories of subunit vector fields. This implies strong maximum and minimum principles when the operator has a family of subunit vector fields satisfying the H\"ormander condition. In particular these results hold for a large class of nonlinear subelliptic PDEs in Carnot groups. We prove also a strong comparison principle for degenerate elliptic equations that can be written in Hamilton-Jacobi-Bellman form, such as those involving the Pucci's extremal operators over H\"ormander vector fields.
\end{abstract}
\tableofcontents

\section{Introduction}
\label{intro}
In this note we investigate the validity of Strong Maximum Principles (briefly, SMP) 
and some Strong Comparison Principles for semicontinuous viscosity subsolutions and supersolutions of fully nonlinear second order PDEs
\begin{equation}
\label{1}
F(x,u,Du,D^{2}u)=0 \quad \text{ in }\Om\ ,
\end{equation}
where $F:\overline{\Om}\times\R^d\times(\R^d\backslash\{0\})\times\Sym_d\to \R$, $\Om$ is an open connected set of $\R^d$ and $\Sym_d$ is the set of $d\times d$ symmetric matrices. 
Our basic assumptions are
\begin{itemize}
\item[(i)] $F$ is \emph{lower semicontinuous} and \emph{proper} in the sense of \cite{CIL}, i.e.
\begin{equation*}
F(x,r,p,X)\leq F(x,s,p,Y)\ ,\quad\text{if } r\leq s\ ,\; Y\leq X\ ;
\end{equation*}
\item [(ii)] (\textit{Scaling}) for some $\phi : (0,1]\to (0,+\infty)
$, $F$ 
satisfies
\begin{equation*}
F(x,\xi s,\xi p,\xi X)\geq\ \phi(\xi)F(x,s,p,X)
\end{equation*}
for all $\xi\in(0,1]$, 
 $s\in[-1,0]$, 
  $x\in\Omega$, $p\in\R^d\backslash\{0\}$, and $X\in\Sym_d$;
\end{itemize}
where $Y\leq X$ means that $X-Y$ is nonnegative semidefinite, the usual ordering  in $\Sym_d$.
Moreover we assume that the operator $F$ is non-degenerate elliptic in the direction of some rank-one matrices identified by the 
 next definition.
\begin{defn}\label{subunit}  $Z\in \R^{d}$ is a generalized \emph{subunit vector} (briefly, SV)
 for $F$ at $x\in\Omega$ if 
\begin{equation*}
\sup_{\gamma>0} F(x,0,p,I-\gamma p\otimes p)> 0 \quad \forall p\in \R^{d} \;\text{ such that } \; Z\cdot p\ne 0 ;
\end{equation*}
$Z:\Omega\rightarrow\R^{d}$ is a \emph{subunit vector field} (briefly, SVF)
if $Z(x)$ is SV for $F$ at $x$ for every $x\in\Omega$. 
\end{defn}
The name is motivated by the the notion introduced  by Fefferman and Phong \cite {FP} for linear operators 
\begin{equation}
\label{lin}
F(x,D^2 u(x)):=-\mathrm{Tr}(A(x)D^{2}u(x)) .
\end{equation}
They call $Z$ a subunit vector for $A$ at $x$ if $A
\geq Z
\otimes Z
$, i.e. 
$$\xi^T A(x)\xi\geq (Z(x)\cdot \xi)^2 \qquad \forall \xi\in\R^d.
$$
It is easy to show that a classical subunit vector is a generalized SV in our sense, and 
that if $Z$ is a SV according to Definition \ref{subunit}, with $F$ linear, then $rZ$ is subunit  for the matrix $A$ 
 for all $r>0$ small enough, see Section \ref{prelim}. 

Our first result concerns the propagation of maxima of a subsolution to \eqref{1} along the trajectories of a subunit vector field.
\begin{thm} \label{teo1}
Assume $F$ satisfies {\upshape(i), (ii)}, and it has a 
 locally Lipschitz subunit vector field $Z$. Suppose $u\in \USC(\Omega)$ is a viscosity subsolution of  \eqref{1} attaining a nonnegative maximum at $x_0\in\Omega$. Then $u(x)=u(x_0)=\max_\Omega u$ for all $x=y(s)$ for some $s\in\R$, where $y'(t)=Z(y(t))$ and $y(0)=x_0$.
\end{thm}
If $F$ has more than one SVF, say a family $Z_i$, $i=1,\dots,m$, we can piece together their trajectories to find a larger set of propagation of the maximum.
It is natural to consider the control system
\begin{equation}
\label{S_0}
y'(t)=\sum_{i=1}^{m}Z_i(y(t))\beta_{i}(t)\ ,
\end{equation}
where the controls $\beta_i$ are measurable functions taking values in a fixed neighborhood of 0. If this system has the property of \emph{bounded time controllability} , namely
\begin{equation}\label{BTC}\tag{BTC}
\text{
$\forall\, x_{0},x_{1}\in\Omega 
\quad \exists \quad$a trajectory $y(\cdot)$ of 
 \eqref{S_0}}
\text{ with $y(0)=x_0, y(s)=x_1,$ 
$y(t)\in\Omega$ 
$\forall\, t\in[0,s]$,}
\end{equation}
then a nonegative maximum of the subsolution $u$ propagates to all $\Omega$, and therefore $u$ is constant. A classical sufficient condition for (BTC), for vector fields smooth enough, is the \emph{H\"ormander condition}  that $Z_1,...,Z_m$ 
and their commutators of any order span $\R^d$ at any point of $\Omega$. Then we have the following
\begin{cor}
[Strong Maximum Principle]
\label{SMaxP}
Assume  {\upshape(i), (ii)}, and the existence of 
subunit vector fields $Z_{i}$, $i=1,...,m$, of $F$ satisfying the H\"ormander condition.
Then any  viscosity subsolution of  \eqref{1} attaining a nonnegative maximum in $\Omega$ is constant.
\end{cor}
This result is a generalization to fully nonlinear equations of the classical maximum principle of Bony \cite{Bony} for smooth subsolutions of linear equations (see also 
\cite{Taira}).

Our main application concerns fully nonlinear subelliptic equations, as defined by Manfredi \cite{MNotes}. 
Given a family $\mathcal{X}=(X_1,...,X_m)$ of $C^{1,1}$ vector fields in $\Om$ 
one defines the intrinsic (or horizontal) gradient and intrinsic Hessian as 
$$
D_{\mathcal{X}}u=(X_1 u,...,X_m u) , \quad (D_{\mathcal{X}}^2 u)_{ij} =X_i (X_j u).
$$
A subelliptic equation has the form
\begin{equation}
\label{subell}
G(x,u,D_{\mathcal{X}}u,(D_{\mathcal{X}}^2u)^*)=0\ ,
\end{equation}
where $Y^*$ is the symmetrized matrix of $Y$ and $G:\overline{\Om}\times\R\times(\R^m\backslash\{0\})\times\Sym_m\to \R$ satisfies at least (i). We assume that   $G$ is \emph{elliptic for any $x$ and $p$ fixed} 
 in the following sense:
\begin{equation}
\label{elli}
\sup_{\gamma>0} G(x,0,q,X-\gamma q\otimes q)> 0 \quad \forall \, x\in \Om,\; q\in \R^{d}, \; q\ne 0 , \; X\in \Sym_m . 
\end{equation}
By rewriting the equation \eqref{subell} in Euclidean coordinates we find an equivalent equation of the form \eqref{1} with $F$ having $X_1,...,X_m$ as subunit vector fields. Then we find the following Strong Maximum Principle for fully nonlinear subelliptic equations:
\begin{cor}
\label{SMPsub}
Assume $G$ verifies {\upshape(i), (ii)}, and  \eqref{elli}, and the vector fields $X_1,...,X_m$
 satisfy the H\"ormander condition.
Then any  viscosity subsolution of  \eqref{subell} attaining a nonnegative maximum in $\Omega$ is constant.
\end{cor}
In Section \ref{sec:sube} we give several examples of operators satisfying the assumptions of this result, including the $m$-Laplacian, the  $\infty$-Laplacian, and Pucci's extremal operators associated to H\"ormander vector fields. Let us recall that the generators of stratified Lie groups, or Carnot groups, satisfy the H\"ormander property. Many examples of such sub-Riemannian structures can be found in \cite{BLU}, the most famous being the Heisenberg group, Example \ref{heis}. Therefore the last Corollary applies to a large number of degenerate elliptic PDEs. 
In Section \ref{sec:appl} we also give applications to Hamilton-Jacobi-Bellman and Isaacs equations.  

Next we make an application to a Strong Comparison Principle
, that is, the following property:

\smallskip
\noindent (SCP) \emph{ if $u$ and $v$ are a sub- and supersolution of \eqref{1} and $u-v$ attains a nonnegative maximum in $\Om$, then $u \equiv v + $constant.}
\smallskip

 \noindent If $\Om$ is bounded the SCP implies the usual (weak) Comparison Principle, namely, $u\leq v$ in $\Om$ if in addition $u\in \USC(\overline\Om)$,  $v\in \LSC(\overline\Om)$, and $u\leq v$ in $\partial\Om$. For a class of equations that can be written in Hamilton-Jacobi-Bellman  form we can show that $w:=u-v$ is a subsolution of a homogeneous PDE $F_0(x,w,Dw, D^2w)=0$ satisfying the SMP, and therefore we deduce immediately the SCP. A model problem is the equation
\begin{equation}
\label{Pucci+H}
\mathcal{M}^{+
}((D_{\mathcal{X}}^2u)^*) + H(x, 
Du) = 0 ,
\end{equation}
where $\mathcal{M}^+
$ denotes the Pucci's maximal operator (see Section \ref{sec:sube} for the definition), $\!\mathcal{X}\!=\!(X_1,...,X_m)$ are H\"ormander vector fields, and $H(x,
p)=\sup_\al\{ p\cdot b^\al(x) + 
 f^\al(x)  \}$ with data $b^\al, 
  f^\al$ bounded and Lipschitz 
 uniformly in $\al$
 .  
Remarkably, this result implies the (weak) Comparison Principle also in some cases for which it was not yet known, see Section \ref{sec:scp}.

\smallskip
The Strong Maximum Principle for elliptic equations goes back to E. Hopf and has a very wide literature, see, e.g., the treatise \cite{GT} and the references therein. We will only mention the papers close enough to our work. 
 The 
seminal contributions on degenerate elliptic linear equations are due to Bony \cite{Bony} and Stroock and Varadhan \cite{SV}: they made the link between the propagation set and, respectively, the set reachable  by a deterministic control system and the support of a diffusion process, for classical solutions. For viscosity subsolutions of uniformly elliptic fully nonlinear equations the SMP was proved by Caffarelli and Cabr\'e \cite{CC} as a consequence of the Harnack inequality. Under lower ellipticity assumptions it was derived in a more direct way in \cite{KK} (in a weaker form) and \cite{BDL1}. Control theoretic and probabilistic descriptions of the propagation set for Hamilton-Jacobi-Bellman equations were given in \cite{BDL2} and  \cite{BDL3}. Our SMP for such equations, Corollary \ref{smp-hjb}, is derived in a simpler way and extends to Isaacs equations, see Section \ref{sec:hji}.
 
 The theory of subelliptic fully nonlinear PDEs began with \cite{MNotes} and \cite{B}, see also \cite{BBM, BCAMS, Wa}. Corollary \ref{SMPsub} seems to be the first Strong Maximum Principle for such equations.
 
 The Strong Comparison Principle for Lipschitz 
  viscosity solutions of uniformly elliptic 
 equations was found by Trudinger \cite{Tru}. There are only a few other results of this kind for fully nonlinear equations: they concern particular PDEs motivated by geometric problems \cite{GO, ML, LW, Capogna}  and are quite different from our Theorem \ref{SCP}. On the other hand the literature on the (weak) Comparison Principle is huge: the results are very general if $F$ is strictly proper (i.e., strictly increasing in $r$) since they include first order equations, see \cite{CIL, BCD}. Under the mere properness (ii), instead,  some ellipticity is needed and the minimal conditions are an open problem, see \cite{Je2, BB, KK, KK2}, and \cite{MNotes, B, BM, Wa, Bi12} for equations involving H\"ormander vector fields, see also the references therein. Our Corollary \ref{WCP} completes the results of \cite{BM}.

As an application of the SMP we will prove in a forthcoming paper the Liouville property for some fully nonlinear equations, extending to the degenerate elliptic case some results of \cite{BC}. 
By the methods of this paper we can also prove SMP and SCP for degenerate parabolic equations, some of these results will appear in a forthcoming paper and in
\cite{TesiAle}.

The paper is organized as follows. In Section \ref{sect:smp} we prove a geometric property of the propagation set of an interior maximum in terms of SV and deduce the connection with the controllability of system \eqref{S_0}, as well as a Hopf boundary lemma. Then we get some strong maximum and minimum principles. Section \ref{sec:appl} presents the applications to some subelliptic nonlinear equations associated to a family of vector fields, to H-J-B and H-J-Isaacs equations, and some other examples. All these results are new, except for the Euclidean case, i.e., when $\mathcal{X}$ is a basis of $\R^d$. Finally, in Section \ref{sec:scp} we prove the Strong Comparison Principle and give some examples.


\section{Strong Maximum and Minimum Principles}
\label{sect:smp}
\subsection{Definitions and preliminaries}
\label{prelim}



We begin by comparing our Definition \ref{subunit} of subunit vector for the operator $F$ 
with the classical one given by Fefferman-Phong for linear operators \eqref{lin}. 
We recall that 
 a vector 
  $Z$ is subunit for $A$ at a point $x$, that we freeze and do not display,  if $A
  \geq Z
\otimes Z(x)$. 
Then 
\begin{equation*}
F(x,0,p,I-\gamma p\otimes p)= 
 -\mathrm{Tr} A + \gamma p \cdot A
 p \geq  -\mathrm{Tr} A + \gamma\sum_{i,j}Z_i Z_j p_j p_i
=-\mathrm{Tr} A + \gamma|Z(x)\cdot p|^{2}
\end{equation*}
which can be made positive for $\gamma$ large enough if $Z\cdot p\ne 0$.
As a partial converse we can prove the following.
\begin{lemma}
If $Z$ is a SV at $x$ for $F$ linear \eqref{lin}, then  $rZ$ is subunit for $A(x)$ for some $r>0$. 
\end{lemma}
\begin{proof}
In view of Definition \ref{subunit}, one easily observes that $Z$ is SV if and only if
\[
\sum_{i,j}a_{ij}p_ip_j=\mathrm{Tr}(Ap\otimes p)>0\text{ for all $p$ such that $p\cdot Z\neq0$.} 
\]
Set $k=\mathrm{rank}(A)$. Then, one may always diagonalise the matrix $A$ in order to have that
\[
a_{ij}=\lambda_i\delta_{ij}\ ,\lambda_i>0\text{ for }i=1,...k\ ,\lambda_i=0\text{ for }i=k+1,...,d\ ,
\]
so the above condition reads
\begin{equation}\label{SubunitEq}
\sum_{i}\lambda_ip_i^2>0\text{ for all $p$ such that }p\cdot Z\neq0\ .
\end{equation}
One can check the following easy characterisation 
 \cite{Taira}: $Z$ is subunit for $A$ if and only if $r Z$ is contained in the following ellipsoid
\[
E:=\left\{\eta\in\R^d:\sum_{i=1}^k\frac{\eta_i^2}{\lambda_i}\leq 1\ ,\eta_{k+1}=...=\eta_d=0\right\}
\]
for some small $r$. Then, 
if $r Z$ does not belong to $E$ there exists a component $Z_j\neq 0$ with $j=k+1,...,d$, since, up to rescaling, the condition $\sum_{i=1}^k\frac{\eta_i^2}{\lambda_i}\leq 1$ is always satisfied. Thus, by taking $p=e_{j}$ it follows that $p\cdot Z\neq 0$, but
$
\sum_{i}\lambda_i p_i^2=0 ,
$
a contradiction with \eqref{SubunitEq}.
\end{proof}

\begin{ex}
It is easy to check, by means of Cauchy-Schwarz inequality, that  the columns of a positive semidefinite matrix $A$ are subunit vectors 
 after multiplication by a sufficiently small 
constant. Moreover, if $A$ can be decomposed as $A=
\sigma\sigma^T$ with $\sigma\in\R^{d\times m}$, then the columns of $\sigma$ are subunit vectors for $A$ 
 (see, e.g., \cite[Example 2.2-2.3]{BDL3}).
\end{ex}

Since equation \eqref{1} 
can be singular at $p=0$, the notion of viscosity solution 
 is slightly weakened 
 with respect to the classical one 
 \cite{CIL}, as follows:

\noindent
{\it a function $u\in\USC(\Om)$ (resp. $\LSC(\Om)$) is a viscosity subsolution (resp. supersolution) of the \eqref{1} in $\Omega$ if, for every $\varphi\in C^2(\Omega)$ and $x$ maximum (resp. minimum) point of $u-\varphi$ such that $D\varphi(x)\neq0$ 
}
\begin{equation*}
F(x,u(x),D\varphi(x),D^2\varphi(x))\leq 0\ (\text{resp.} \geq 0)\ .
\end{equation*}
From now on all sub- and supersolutions will be meant in the viscosity sense.

We define the \emph{Propagation set}  of a  viscosity subsolution $u$ 
 of \eqref{1} attaining a nonnegative maximum at $x\in\Om$ as 
 $$
 \Prop(x,u):=\{ y\in\Omega : u(y)=u(x)= \max_\Om u\} .
 $$ 

We will need the notion of generalized exterior normal, also called Bony normal or 
{proximal normal} (see, e.g., \cite{Bony} or \cite[Definition 2.17]{BCD}):

\noindent a unit vector $\nu$ is a \emph{generalized exterior normal} to a nonempty set $K\subseteq\R^d$ at $z\in\partial K$ if 
there is a ball outside $K$ centered at $z+t\nu$ 
for some $t>0$ touching $\overline{K}$ precisely at $z$, i.e. 
$\overline{B}(z+t\nu,t)\cap\overline{K}=\{z\}$. Then we 
 write that  $\nu\bot K$ at $z$, and we use also the notation
\begin{equation*}
K^*:=\{z\in \partial K: \text{there exists } \nu\bot K\text{ at }z\}\ .
\end{equation*}

As in the classical paper of Bony \cite{Bony} we will use a 
geometric characterisation of invariant sets for the control system \eqref{S_0}, that we recall next.  We consider as admissible the 
control functions $\beta=(\beta_1,...,\beta_m): [0,+\infty) \to \R^m$ in the set
\[
\mathcal{B}:=\{\beta : \sum_{i=1}^{m}\beta_{i}^{2}(t)\leq1\text{ and }
\text{$\beta_i$ is measurable  $\forall \, i=1,...,m$}\} , 
\]
and denote with   $y_x(\cdot,\beta)$ the solution of the system \eqref{S_0} with initial condition $y(0)=x$, which exists at least locally if the vector fields $Z_i : \overline\Om \to \R^d$ are locally Lipschitz. 

\noindent A 
set $K\subseteq \overline\Om$ is \emph{invariant for the system \eqref{S_0}} if
 for all $x\in K$, $\beta\in\mathcal{B}$ 
 and $\tau>0$ such that the solution $y_x(\cdot,\beta)$ exists in $[0,\tau)$, we have $y_x(t,\beta)\in K$ for all $t\in[0,\tau)$.
\begin{thm}\label{invariance}
Let $Z_i : \overline\Om \to \R^d$ be locally Lipschitz and $ K\neq\emptyset$ 
 be a relatively closed subset of $\Om$. 
 If for all $x\in K^*\cap\Om$ and for all $\nu\bot K$ at $x$  
\begin{equation}\label{orto}
 Z_i(x)\cdot\nu=0 \quad \forall \, i=1,\dots, m,
\end{equation}
  then $K$ is invariant for \eqref{S_0}.
\end{thm}
\begin{proof}
We can repeat the proof of \cite[Theorem 2.1]{BDL2}, which combines the classical result for $\Om =\R^d$ with a localization argument. Then it is easy to see that it is enough to assume \eqref{orto} at points $x\in\partial K\cap \Om$.
\end{proof}

%
\subsection{Propagation of maxima}
\label{sec:max}
We first give a technical result providing a crucial geometric property of the propagation set.

\begin{prop}\label{prop11}
Let $u$ 
be a viscosity subsolution of \eqref{1} that achieves a nonnegative maximum at $x
\in\Omega$. Assume that {\upshape (i)-(ii)} hold and $F$ has a subunit vector field as in Definition \ref{subunit}.
Then $K:=\Prop(x,u)$ is such that for every $z\in 
K^*\cap\Omega$ and for every $\nu\bot K$ at $z$ we have $Z\cdot\nu=0$ for every subunit vector of $F$ at $z$. 
\end{prop}
\begin{proof} 
We fix $z\in \partial K\cap\Omega$ and $\nu\bot K$ at $z$. 
Arguing by contradiction, we assume  there exists a subunit vector $\bar{Z}$ at $z$ such that $\bar{Z}\cdot\nu\neq0$.  By definition of normal we can take  $R>0$ and $y=z+R\frac{\nu}{|\nu|}$ such that $B(y,R)\subseteq\Omega\backslash K$. We divide the proof in two steps.\\
\par\smallskip
\textit{Step 1}. We claim that there exist $r>0$ and a function $v\in C^{2}(\R^d)$ such that
\begin{equation*}
F(x,v(x),Dv(x),D^2 v(x))\geq C>0\text{ for every }x\in B(z,r)\ ,
\end{equation*}
with the properties $v(z)=0$, $-1<v<0$ in $B(y,R)$ and $v>0$ outside $B(y,R)$.\\
To see this, consider
\begin{equation}\label{v}
v(x)=e^{-\gamma R^{2}}-e^{-\gamma|x-y|^2}\ .
\end{equation}
Note that $v\equiv0$ on $\partial B(y,R)$ (which gives $v(z)=0$) and $v>0$ outside $B(y,R)$. Moreover $-1<v<0$ in $B(y,R)$. By direct computations we have
\begin{equation*}
Dv(x)=2\gamma e^{-\gamma|x-y|^2}(x-y)
\end{equation*}
and
\begin{equation*}
D^2 v(x)=2\gamma e^{-\gamma|x-y|^2}(I-2\gamma(x-y)\otimes(x-y))\ .
\end{equation*}
Now, using that $z-y=-\nu$ and the scaling property (ii) we have
\begin{multline}\label{pos1}
F(z,v(z),Dv(z),D^2 v(z))=F(z,0,2\gamma e^{-\gamma R^2}(-\nu),2\gamma e^{-\gamma R^2}(I-2\gamma\nu\otimes\nu))\\
\geq \phi(2\gamma e^{-\gamma R^2})F(z,0,-\nu,I-2\gamma\nu\otimes\nu)\ .
\end{multline}
By the definition of subunit vector at $z$ and $\bar{Z}\cdot\nu\ne 0$ we obtain
\begin{equation*}\label{posi2}
F(z,0,-\nu,I-2\gamma\nu\otimes\nu)> 0 
\end{equation*}
for some $\gamma
>0$. 
Then \eqref{pos1} and $ \phi(\xi)>0$ for all $\xi>0$ give
$F(z,v(z),Dv(z),D^2 v(z))>0$. Since $F$ is lower semicontinuous 
 we can conclude that there exists $r>0$ such that
\begin{equation}
\label{cont}
F(x,v(x),Dv(x),D^2 v(x))\geq C>0\text{ for every }x\in B(z,r)\ .
\end{equation}

\textit{Step 2.} We claim now that there exists $\epsilon>0$ such that $u(x)-u(z)\leq\epsilon v(x)$ in $X:=B(z,r)\cap B(y,R)$.\\
Let us choose $\epsilon>0$ small enough such that $u(x)-u(z)\leq \epsilon v(x)$ for every $x\in\partial X$. To prove that the inequality holds on the whole $X$, suppose by contradiction that there exists $\bar{x}\in X$ such that $u(\bar{x})-u(z)-\epsilon v(\bar{x})=\max_{X}(u-u(z)-\epsilon v)>0$. Since $\epsilon v$ is smooth in $\R^d$, using that $u-u(z)$ is a viscosity subsolution of \eqref{1}  and the scaling property (ii) we get
\begin{equation*}
\phi(\epsilon)F(\bar{x}, v(\bar{x}), Dv(\bar{x}), D^2 v(\bar{x}))\leq F(\bar{x},\epsilon v(\bar{x}),\epsilon Dv(\bar{x}),\epsilon D^2 v(\bar{x}))\leq 0
\end{equation*}
which contradicts \eqref{cont} because $\phi>0$.
\\
\smallskip
Then $u(x)-\epsilon v(x)\leq u(z)$ and $u(z)-\epsilon v(z)=u(z)$ since $v(z)=0$. Therefore the function $\Phi(x):=u(x)-\epsilon v(x)$ has a maximum at $z$ in $X$. Moreover in $B(z,r)\backslash X$ we have $v\geq0$ and $u(x)-\epsilon v(x)\leq u(x)\leq u(z)$. As a consequence the function $\Phi(x)$ has a maximum in $B(z,r)$ at $z$. Since $\epsilon v\in C^{\infty}(\R^d)$, $F$ is proper, using also the definition of viscosity subsolution and (ii) we get
\begin{equation*}
\phi(\epsilon)F(z,v(z),Dv(z),D^2 v(z))\leq F(z,u(z),\epsilon Dv(z),\epsilon D^2 v(z))\leq 0,
\end{equation*}
 a contradiction with \eqref{cont}.
\end{proof}
Our main result is the following, containing Theorem \ref{teo1} as a special case.
\begin{thm}
\label{teo:main}
Let $u$ 
be a viscosity subsolution of \eqref{1} that achieves a nonnegative maximum at $x
\in\Omega$. Assume that {\upshape (i)-(ii)} hold and $F$ has  locally Lipschitz continuous subunit vector fields $Z_{i} : \overline\Om\to \R^d$, $i=1,...,m$.
Then $\Prop(x,u)$ contains all the points reachable by the system \eqref{S_0} starting at $x$, i.e., if $y=y_x(t,\beta)$ for some $ t>0, \beta\in\mathcal{B}$, then $y\in \Prop(x,u)$. 
\end{thm}
\begin{proof}
If $\Prop(x,u)=\Om$ the conclusion is true. Otherwise, for all $z\in\partial \Prop(x,u)\cap \Om$  Proposition \ref{prop11} implies $Z_i(z)\cdot\nu=0$ for all  $\nu\bot \Prop(x,u)$ at $z$ and $i=1,\dots,m$. Then Theorem \ref{invariance} ensures the invariance of $\Prop(x,u)$ for the system \eqref{S_0}, and therefore all trajectories starting at $x$ remain forever in $\Prop(x,u)$.
\end{proof}
\begin{cor}[Strong Maximum Principle]
\label{SMaxP2} 
In addition to the assumptions of Theorem \ref{teo:main} suppose the system \eqref{S_0} satisfies the bounded time controllability property \eqref{BTC}. Then 
$u$ is constant. 
\end{cor}
\begin{proof}
If (BTC) holds then any point of $\Om$ is reachable by the system \eqref{S_0} starting at $x$. Then Theorem \ref{teo:main} gives $\Prop(x,u)=
\Omega$.
\end{proof}
Before proving Corollary \ref{SMaxP} we recall that the classical H\"ormander condition requires that 

\smallskip
\noindent (H) \emph {the vector fields $Z_{i}$
, $i=1,...,m$, are $C^\infty$ and the 
 Lie algebra generated by them has full rank $d$ at each point of $\Om$.}
 \smallskip
 
 \noindent
  The smoothness requirement on $Z_{i}$ can be reduced to $C^k$ for a suitable $k$ and considerably more if the Lie brackets are interpreted in a generalized sense, see \cite{FR} and the references therein.
\begin{proof} [Proof of Corollary \ref{SMaxP}]
By the classical Chow-Rashevskii theorem in sub-Riemannian geometry and its control-theoretic version (see, e,g, \cite[Lemma IV.1.19]{BCD}), for any $z\in\Om$ the set of points reachable from $z$ by the system contains a neighborhood of $z$. Since $u\in \USC(\Om)$, $K=\Prop(x,u)=\{y\in \Om : u(y)
 = \max u\}$ is relatively closed. 
Then $\Om$ connected implies that either $K=\Om$ or $K$ is not relatively open. In the latter case there would be $z\in K$ with no neighborhood contained in $K$, a contradiction with Theorem \ref{teo:main}. Then $K=\Om$.
\end{proof}
\begin{rem}
Note that the existence of a SV at $x$ for $F$ and the scaling property (ii) imply
 \[
 \limsup_{(s,p,X)\to (0,0,0)} F(x,s, p, X) \geq 0 \,,
 \]
 a weaker condition than  $F(x,0,0,0) \geq 0$   used in \cite{KK}.
\end{rem}
\begin{rem}
\label{gen scal}
It is easy to see from the proof of Proposition \ref{prop11} that the function $\phi$ in the scaling property (ii) can be allowed to depend also on $x, s, p$, and $X$. What is really needed is that $F(x,s,p,X)>0$ implies $F(x,\xi s,\xi p,\xi X)>0$ for all $\xi\in(0,1]$ and all $x,s,p,X$.
\end{rem}
\begin{rem}
In all the previous results the 
scaling assumption (ii) on  $F$  can be avoided 
if there is $\tilde{F}$ satisfying all conditions 
and approximating $F$  in the sense that
\begin{equation*}
F(x,\epsilon s, \epsilon p,\epsilon X)\geq\tilde{F}(x,\epsilon s,\epsilon p,\epsilon X)+\phi(\epsilon)\psi(\epsilon)\ 
\end{equation*}
with $\lim_{\epsilon\rightarrow0^+}\psi(\epsilon)=0$. Indeed, in the proof of Proposition \ref{prop11} one can see that \eqref{cont} still holds 
under this assumption (cfr. \cite{BDL1}). 
\end{rem}
We end with section with
\begin{lemma}[Hopf boundary lemma] 
Let $U\subseteq\Omega$ be an open set, $x_0\in\partial U$, $u\in\mathrm{USC}(U\cup\{x_0\})$ be a viscosity subsolution of \eqref{1} in $U$ such that
\begin{itemize}
\item[(a)] $u(x_0)>u(x)$ for every $x\in U$ and $u(x_0)\geq0$;
\item[(b)] there exists a ball $B:=B(y,R)$ such that $B\subseteq U$ and $\overline{B}\cap \partial U=\{x_0\}$. 
\end{itemize}
Assume that $F$ satisfies (i)-(ii) and there exists a SV $Z$ for $F$ such that $p:=x_0-y$ satisfies $p\cdot Z\neq 0$. Then, for any $w\in\R^d$ such that $w\cdot p<0$, we have
\[
\limsup_{\tau\to 0^+}\frac{u(x_0+\tau w)-u(x_0)}{\tau}<0
\]
\end{lemma}
\begin{proof}
 As in Step 1 of Proposition \ref{prop11} we define $v$ as in
 \eqref{v}, which turns out to be a strict classical supersolution in $\overline{X}:=B\cap B(x_0,r)$ for a suitably small $r>0$ because $p\cdot Z\neq 0$. Then, arguing as in Step 2 of Proposition \ref{prop11} one proves that $u(x)-u(x_0)\leq \epsilon v(x)$ for every $x\in \overline{X}$.
 To conclude, it is then sufficient to observe that, for any $w\in\R^d$ such that $w\cdot p<0$, one has
\[
\limsup_{\tau\to 0^+}\frac{u(x_0+\tau w)-u(x_0)}{\tau}\leq \eps Dv(x_0)\cdot w=2\gamma e^{-\gamma|x_0-y|^2}p\cdot w<0\ .
\]
\end{proof}
\subsection{Propagation of minima}
\label{sec:min}
Various  Strong Minimum Principles for (viscosity) supersolutions of \eqref{1}  can be easily derived from the results of the previous section by recalling that 
$v\in\LSC(\Om)$ is a supersolution of \eqref{1}  if and only if $u=-v$ is a subsolution of 
\[
-F(x,-u,-Du, -D^2v) = 0 \quad\text{ in }\; \Om .
\]
Therefore one can read properties of the minima of $v$ from the preceding results by applying them to $u$ and 
\[
F^-(x, r, p, X) := -F(x, -r, -p, -X) .
\]
Let us make explicit the assumptions on $F$ that imply a Strong Minimum Principle. First we
 replace (i)-(ii) by
\begin{itemize}
\item[(i')] $F:\Omega\times\R^d\times\R^d\backslash\{0\}\times\Sym_d\to \R$ is upper semicontinuous and proper.
\item [(ii')] 
For some $\phi>0$ the operator satisfies $F(x,\xi s,\xi p,\xi X)\leq\ \phi(\xi)F(x,s,p,X)$ for all $\xi\in(0,1]$ and $s\in [0, 1]$.
\end{itemize}
Moreover, a vector $Z$ is a subunit vector for $F^-$ at $x$ if and only if 
\begin{equation}
\label{subunit-}
\inf_{\gamma>0} F(x,0,p,\gamma p\otimes p-I)< 0 \quad \forall p\in \R^{d} \;\text{ such that } \; Z\cdot p\ne 0 .
\end{equation}
Now we can easily get the following properties of minima.
\begin{cor}\label{cor:min}
Let $v\in\LSC(\Om)$ be a viscosity supersolution of \eqref{1} that achieves a nonnegative minimum at $x
\in\Omega$. Assume that {\upshape (i')-(ii')} hold and  $Z_{i} : \overline\Om\to \R^d$, $i=1,...,m$, are  locally Lipschitz 
 subunit vector fields of $F^-$, i.e., at each $x\in \Om$   $Z_{i}(x)$ verifies \eqref{subunit-}.
Then 
$v(y)=v(x)=\min_\Om v$ for all points $y$ reachable by the system \eqref{S_0} starting at $x$.
\end{cor}
\begin{cor}[Strong Minimum Principle]\label{cor:min} 
In addition to the assumptions of Corollary \ref{cor:min} suppose the system \eqref{S_0} satisfies the bounded time controllability property \eqref{BTC}. Then 
$v$ is constant. 
This holds in particular if the fields $Z_i$, $i=1,...,m$, verify the H\"ormander condition.
\end{cor}

\section{Some applications}
\label{sec:appl}

\subsection{
Fully Nonlinear Subelliptic Equations}
\label{sec:sube}
Our main application concerns fully nonlinear subelliptic equations. In this framework one is given a family $\mathcal{X}=(X_1,...,X_m)$ of $C^{1,1}$ vector fields defined in $\overline{\Omega}$. 
 The intrinsic gradient and intrinsic Hessian  are defined as $D_{\mathcal{X}}u=(X_1 u,...,X_m u)$ and $(D_{\mathcal{X}}^2 u)_{ij} =X_i (X_j u)$. 
After choosing a base in Euclidean space we write $X_j=\sigma^j\cdot D$, with $\sigma^j:\overline{\Omega}\rightarrow\R^d$, and $\sigma=\sigma(x)=[\sigma^1(x),...,\sigma^m(x)]\in\R^{d\times m}$. 
Then
\begin{equation*}
D_{\mathcal{X}}u=\sigma^T Du=(\sigma^1\cdot Du,...,\sigma^m\cdot Du)
\end{equation*}
and
\begin{equation*}
X_i (X_j u)=(\sigma^T D^2 u\ \sigma)_{ij} +(D\sigma^j\ \sigma^i)\cdot Du\ .
\end{equation*}
Denote by $Y^*$ the symmetrized matrix of $Y$. By the chain rule (see, e.g., \cite[Lemma 3]{BBM}) one can obtain that for $u\in C^2$
\begin{equation*}
(D_{\mathcal{X}}^2 u)^*=\sigma^T D^2 u\sigma+
g(x,Du)\ ,
\end{equation*}
where the correction term $g$ is
\begin{equation*}
(g(x,P))_{ij}=\frac12[(D\sigma^j\ \sigma^i)\cdot p+(D\sigma^i\ \sigma^j)\cdot p]\ .
\end{equation*}
Then the subelliptic equation \eqref{subell} 
can be written as
\begin{equation}
\label{subell2}
G(x,u,\sigma^T(x) Du,\sigma^T(x) D^2 u\sigma(x)+g(x,Du))=0\ ,
\end{equation}
which is of the form \eqref{1} if we define
\begin{equation}\label{euclidean}
F(x,r,p,X):= G(x,r,\sigma^T(x) p,\sigma^T(x)X\sigma(x)+g(x,p))\ .
\end{equation}
\begin{lemma} 
If $G$ satisfies properties {\upshape (i), (ii)} and \eqref{elli} of Section \ref{intro}, then $F$ satisfies properties {\upshape (i)} and {\upshape (ii)} and  the vector fields $
\sigma^i$ 
 are subunit 
  for $F$ in the sense of Definition \ref{subunit}.
\end{lemma}
\begin{proof} (i) holds 
because $X\leq Y$ implies 
$\sigma^T(x)X\sigma(x)\leq \sigma^T(x)Y\sigma(x)$, so  $F$ is proper.

  (ii) holds for $F$ if it does for $G$ because $g(x,p)$ is positively 1-homogeneous in the variable $p$.
  

To prove that any $X_i$ is SV for $F$ we use property  \eqref{elli} of  $G$ with $q=\sigma^T(x) p$, $X=\sigma^T\sigma + g$ to get
\begin{
equation*}
F(x,0,p,I-\gamma p\otimes p)=
G(x,0,\sigma(x)^T p,\sigma^T(x)I\sigma(x)-\gamma(\sigma^T(x)p)\otimes (\sigma^T(x) p)+g(x,p)) > 0
\end{
equation*}
for some $\gamma>0$ if $\sigma^i(x) \cdot p\ne 0$. 
\end{proof} 
%
This Lemma and Theorem \ref{teo:main} give the following propagation of maxima and SMP.
\begin{cor}
\label{Propa:sub}
Assume $G$ verifies {\upshape(i), (ii)}, and  \eqref{elli}, and  let $u$ be a subsolution of \eqref{subell} or, equivalently, \eqref{subell2}, attaining a maximum at $x\in\Om$.
Then $\Prop(x,u)$ contains all the points reachable from $x$ by the system \eqref{S_0} with $Z_i = X_i$. In particular if the property \eqref{BTC} holds for such system then $u$ is constant.
\end{cor}
From this we get immediately the Strong Maximum Principle for subelliptic equations with the H\"ormander condition, Corollary \ref{SMPsub}, as in the proof of Corollary \ref{SMaxP}.
\begin{ex}
A very simple example  in $\R^2$ of vector fields that fail to span  all $\R^2$ at some point but satisfy the H\"ormander condition are the Grushin vector fields, namely, 
\[
\sigma_{\mathcal{G}}(x)=\begin{pmatrix}1 & 0\\0 & x_1\end{pmatrix} \ .
\]
In this case the symmetrized horizontal hessian is given by
\[
(D^2_{\mathcal{X}}u)^*=\sigma^T_{\mathcal{G}}(x)D^2u\sigma_{\mathcal{G}}(x)+g(x,Du)=\begin{pmatrix} u_{x_1x_1}& x_1u_{x_1x_2}+\frac{u_{x_2}}{2}\\ x_1u_{x_1x_2}+\frac{u_{x_2}}{2}& x_1^2u_{x_2x_2}\end{pmatrix} \ .
\]
\end{ex}
\begin{ex}
\label{heis}
The most studied examples of vector fields satisfying the H\"ormander condition are the generators of a Carnot group: see the treatise \cite{BLU} for a comprehensive introduction and for the theory of linear subelliptic equations in such groups. The 
simplest prototype of Carnot group is the  Heisenberg group $\mathbb{H}^1$ in $\R^3$ whose generators are 
\[
\sigma_{\He^1}(x)=\begin{pmatrix}1 & 0\\0 & 1\\ 2x_2 & -2x_1\end{pmatrix}\ .
\]
Here the correction term of the Hessian is $g\equiv 0$, and this occurs for all groups of step 2. An example of Carnot group of step 3 where $g(x,p)\neq 0$ is the Engel group, 
see e.g. \cite[Example 3]{BBM}.

\end{ex}
Next we list some examples of equations 
 of the form 
\begin{equation}\label{modeleq}
c(x)|u|^{k-1}u-a(x)E(
D_{\mathcal{X}}u,(D^2_{\mathcal{X}}u)^*)=0
\end{equation}
where we assume $E : 
 \R^d\backslash\{0\}\times\R^{m\times m}$ is positively homogeneous of degree $\alpha\geq 0$
 , $c, a$ are continuous and satisfy
\begin{equation}\label{coeff}
c\geq0\ ,\, a>0\ ,
\; \text{ and }\; 
\text{ either }c=0\text{ or }\alpha\leq k\ ,k>0\ .
\end{equation}
We give some examples of operators $E$ 
for which the SMP and Strong Minimum Principle for equation \eqref{modeleq} are known to hold in the Euclidean case, i.e., if the fields $\mathcal{X}$ are the canonical basis of $\R^d$, see \cite{BDL1}. Our contribution is that they hold for H\"ormander vector fields as well.
\begin{ex}
The 
 \emph{subelliptic $\infty$-Laplacian} \cite{B, BCAMS, Wa} is
\[
-\Delta_{\mathcal{X},\infty}u=- D_{\mathcal{X}}u \cdot (D^2_{\mathcal{X}}u)^*D_{\mathcal{X}}u \,
\]
where $E=-p\cdot Xp$ is homogeneous of degree $\alpha=3$ and 
 \eqref{elli} 
 is satisfied because
\[
E(q,X-\gamma q\otimes q)=-q\cdot Xq+\gamma |q|^4\ .
\]
Note that the associated operator $F$ satisfies also the condition \eqref{subunit-}. Then the equation \eqref{modeleq} with $E$ the $\infty$-Laplacian satisfies both the SMP and the Strong Minimum Principle.
\end{ex}
\begin{ex}
A generalization of the previous example (considered in \cite{BMartin2} for the evolutive case) is
\begin{equation*}
-\Delta_{\mathcal{X},\infty}^hu=-|D_{\mathcal{X}}u|^{h-3}(D^2_{\mathcal{X}}u)^*D_{\mathcal{X}}u\cdot D_{\mathcal{X}}u
\end{equation*}
with $h\geq 0$, where $E$ is homogeneous of degree $h$ and 
satisfies  \eqref{elli} because
\[
E(q,X-\gamma q\otimes q)=E(q,X)+\gamma|q|^{h+1}\ .
\]
\end{ex}
\begin{ex}
The 
 \emph{subelliptic $m$-Laplacian}, $m>1$, is
\[
-\Delta_{\mathcal{X},m}u:=-\mathrm{div}_{\mathcal{X}}(|D_{\mathcal{X}}u|^{m-2}D_{\mathcal{X}}u)
 = -(|D_{\mathcal{X}}u|^{m-2}\Delta_{\mathcal{X}}u+(m-2)|D_{\mathcal{X}}u|^{m-4}\Delta_{\mathcal{X},\infty}u) \,
 \]
where $\Delta_{\mathcal{X}}u:=\text{Tr} (D^2_{\mathcal{X}}u)$ is the sub-Laplacian. Here $E$ is homogeneous of degree $\alpha=m-1$ and 
 \eqref{elli} holds because
\[
E(q,X-\gamma q\otimes q)=E(q,X)+\gamma|q|^m(m-1) .
\]
Similarly one checks \eqref{subunit-}. Recently the SMP and a Strong Comparison Principle  were proved in \cite{Capogna} for weak $C^1$ solution of similar equations involving the subelliptic $m$-Laplacian. 

Since the $m$-Laplacian is in divergence form the natural notion of solution for $-\Delta_{\mathcal{X},m}u = 0$ is variational. The equivalence of solutions in Sobolev spaces with viscosity solutions was shown by Bieske \cite{Bi12} in Carnot groups. For this homogeneous equation the SMP can also be deduced from the Harnack inequality, see the references in \cite{Capogna}.
\end{ex}
\begin{ex}
For fixed $0<\lambda\leq\Lambda$, the \textit{Pucci's extremal operators} on symmetric matrices $M\in\Sym_{m}$ are
\begin{equation}
\label{7}
\mathcal{M}^{+}
(M)
:= -\lambda\sum_{e_{k}>0}e_{k}-\Lambda\sum_{e_{k}<0}e_{k} = \sup
 \{ -\mathrm{Tr}(AM) : A\in\Sym_{d} ,\, \lambda I\leq A \leq \Lambda I \} \,
 \end{equation}
\begin{equation}
\label{6}
\mathcal{M}^{-}(M)=-\Lambda\sum_{e_{k}>0}e_{k}-\lambda\sum_{e_{k}<0}e_{k} = \inf\{ -\mathrm{Tr}(AM) : A\in\Sym_{d} ,\, \lambda I\leq A \leq \Lambda I \}  .
\end{equation}
They are $1$-homogeneous and satisfy \eqref{elli} because
\[
\mathcal{M}^{+}(X-\gamma q\otimes q)\geq \mathcal{M}^{-}(X-\gamma q\otimes q)\geq \mathcal{M}^{-}(X)-\lambda\gamma |q|^2 .
\]
If we take a \emph{subelliptic Pucci's operator} 
$
E(
(D^2_{\mathcal{X}}u)^*)=\mathcal{M}^{+}((D^2_{\mathcal{X}}u)^*) 
$ 
then the equation \eqref {modeleq} satisfy the SMP and the Strong Minimum principle, and the same holds if $\mathcal{M}^{+}$ is replaced by  $\mathcal{M}^{-}$.
\end{ex}
%

\subsection{
Hamilton-Jacobi-Bellman Equations}
\label{sec:hjb}
We are given a family of linear degenerate elliptic operators
\begin{equation}
\label{linear}
L^{\al} u:= -\Tr(A^\al(x) D^2u) - b^\al(x)\cdot Du + c^\al(x) u \,
\end{equation}
where the parameter $\al$ takes values in a given set, $A^\al(x)\geq 0$ and  $c^\al(x)\geq 0$ for all $x$ and $\al$.
The H-J-B operators are 
\begin{equation}
\label{hjb-hom}
F_s(x,u,Du, D^2u) := \sup_\al L^{\al} u \, , \qquad F_i(x,u,Du, D^2u) := \inf_\al L^{\al} u \,   
\end{equation}
and we assume that $F_s(x, r,p, X), F_i(x, r,p, X)$  are finite and continuous for all entries  $(x, r,p, X)\in \overline{\Om}\times\R^d\times
\R^d
\times\Sym_d$. 
They are clearly proper and positively $1$-homogeneous. 
We can characterise the subunit vectors of these operators as follows.
\begin{lemma}
\label{lem:hjb}
Let $Z\in\R^d$ and $x\in\Om$.

\noindent i) $Z$ is SV for $F_i$ at $x$ if and only if $Z$ is subunit for all the matrices $A^\al(x)$,  i.e., $A^\al(x)\geq Z\otimes Z$ for all $\al$;

\noindent ii)   $Z$ is SV for $F_s$ at $x$ if   there exists $\bar\al$ such that $Z$ is subunit for the matrix $A^{\bar\al}(x)$. 
\end{lemma}
\begin{proof} \textit{i)}  First suppose $A^{\al}(x)\geq Z\otimes Z$ for all $\al$. Then, for $p\cdot Z\ne 0$ and $\gamma$ large enough,
\begin{multline*}
F_i(x, 0,p, I-\gamma p\otimes p) = \inf_\al \{ -\Tr A^\al(x)  +\gamma p\cdot A^\al(x) p - b^\al(x)\cdot p \} \\
\geq \inf_\al \{ -\Tr A^\al(x)   - b^\al(x)\cdot p \}  +\gamma |Z\cdot p|^2 > 0\, .
\end{multline*}

Viceversa, suppose $Z$ is not a subunit vector of $A^{\bar\al}(x)$. Then there exist $\bar p$ such that $\bar p\cdot Z\ne 0$ and $\bar p\cdot A^{\bar\al}(x) \bar p=0$. Then, for any $\eta\in \R$  and $\gamma >0$
\[
F_i(x, 0,\eta\bar p, I- \gamma \eta^2\bar p\otimes \bar p) \leq -\Tr A^{\bar\al}(x)  -  \eta b^{\bar\al}(x)\cdot \bar p \leq  -  \eta b^{\bar\al}(x)\cdot \bar p \,.
\]
But the right hand side is $\leq 0$ 
by choosing  $\eta = \sign( b^{\bar\al}\cdot \bar p)$, and so 
$Z$ is not SV for $F_i$.

\smallskip
\textit{ii)}
Suppose $A^{\bar\al}(x)\geq Z\otimes Z$. Then, for $p\cdot Z\ne 0$ and $\gamma$ large enough
\begin{multline*}
F_s(x, 0,p, I-\gamma p\otimes p) = \sup_\al \{ -\Tr A^\al(x)  +\gamma p\cdot A^\al(x) p - b^\al(x)\cdot p \} \\
 \geq -\Tr A^{\bar\al}(x) + \gamma |Z\cdot p|^2- b^{\bar\al}(x) \cdot p > 0 \,.
\end{multline*}
 \end{proof}
The results of sections \ref{sec:max} and \ref{sec:min} combined with this Lemma give informations on the sets of propagation of maxima and minima of sub- and supersolutions. This was studied in detail in the papers of the first author and Da Lio \cite{BDL2, BDL3} using also tools from diffusion processes and differential games. Therefore we only point out explicitly a SMP for the concave H-J-B operator $F_i$ that we will exploit in Section \ref{sec:scp}. Its proof is an immediate consequence of Corollary \ref{SMaxP2} and Lemma \ref{lem:hjb}, and therefore it is more direct  than the one in \cite{BDL3}. We also give a Strong Minimum Principle for the convex operator $F_s$ following from Corollary \ref{cor:min}.

\begin{cor}
\label{smp-hjb}
Assume  $Z_{i} : \overline\Om\to \R^d$, $i=1,...,m$, are  locally Lipschitz 
vector fields such that 
$$
A^\al (x)\geq Z_{i}(x)\otimes Z_{i}(x) \quad\text{for all }  \al , \,i, \text{ and } x\,,
$$ 
and the system \eqref{S_0} satisfies the bounded time controllability property \eqref{BTC}.
Then 

\noindent i) any subsolution 
 of $\;\inf_\al L^{\al} u=0\;$ attaining a maximum in $\Om$ is constant,

\noindent ii) any supersolution 
 of $\;\sup_\al L^{\al} u=0\;$ attaining a minimum in $\Om$ is constant.
\end{cor}

%
%
\subsection{
Hamilton-Jacobi-Isaacs Equations}
\label{sec:hji}
Now we are given a two-parameter family of linear degenerate elliptic operators
\[ 
L^{\al, \beta} u:= -\Tr(A^{\al,\beta}(x) D^2u) - b^{\al,\beta}(x)\cdot Du + c^{\al,\beta}(x) u \,
\]
where the parameters $\al, \beta$ take values in two given sets, $A^{\al,\beta}(x)\geq 0$ and  $c^{\al,\beta}(x)\geq 0$ for all $x$, $\al, \beta$.
The Hamilton-Jacobi-Isaacs (briefly, H-J-I) operators are 
\[
F_-(x,u,Du, D^2u) := \sup_\beta \inf_\al L^{\al, \beta} u \, , \qquad F_+(x,u,Du, D^2u) :=  \inf_\al  \sup_\beta L^{\al, \beta}u \,  
  \]
and we assume that $F_-(x, r,p, X), F_+(x, r,p, X)$  are finite and continuous for all entries  $(x, r,p, X)\in \overline{\Om}\times\R^d\times
\R^d
\times\Sym_d$. 
They are clearly proper and positively $1$-homogeneous. 
We can find 
 subunit vectors of these operators following the arguments of Lemma \ref{lem:hjb}.
\begin{lemma}
\label{lem:hji}
Let $Z\in\R^d$ and $x\in\Om$.

\noindent i) $Z$ is SV for $F_-$ at $x$ if 
there exists $\bar \beta$ such that  
 $A^{\al,\bar\beta}(x)\geq Z\otimes Z$ for all  $\al$;

\noindent ii)   $Z$ is SV for $F_+$ at $x$ if  for all  $\al$ there exists $\beta(\al)$ such that 
$A^{\al,\beta(\al)}(x)\geq Z\otimes Z$.
\end{lemma}
Then we get 
 the following SMP for the H-J-I equations.
\begin{cor}
Assume  $Z_{i} : \overline\Om\to \R^d$, $i=1,...,m$, are  locally Lipschitz 
vector fields such that  the system \eqref{S_0} satisfies the bounded time controllability property \eqref{BTC}. Then

\noindent i) if there exists $\bar \beta$ such that
$$
A^{\al,\bar\beta}(x)\geq Z_{i}(x)\otimes Z_{i}(x) \quad\text{for all }  \al , \,i, \text{ and } x\,,
$$ 
then 
any 
subsolution 
 of $\;\sup_\beta \inf_\al L^{\al, \beta}u=0\;$ attaining a maximum in $\Om$ is constant;
 
 \noindent ii)  
 if  for all  $\al$ there exists $\beta(\al)$ such that 
$$
A^{\al,\beta(\al)}(x)\geq Z_{i}(x)\otimes Z_{i}(x)  \quad\text{for all } 
i \text{ and } x\,,
$$
then any 
subsolution 
 of $\; \inf_\al \sup_\beta L^{\al, \beta}u=0\;$ attaining a maximum in $\Om$ is constant.
\end{cor}
Sufficient conditions for the Strong Minimum Principle can be easily found in the same way, as follows.
\begin{cor}
Assume  $Z_{i} : \overline\Om\to \R^d$, $i=1,...,m$, are  locally Lipschitz 
vector fields such that  the system \eqref{S_0} satisfies the bounded time controllability property \eqref{BTC}. Then

\noindent i)
 if  for all  $\beta$ there exists $\al(\beta)$ such that 
$$
A^{\al(\beta),\beta}(x)\geq Z_{i}(x)\otimes Z_{i}(x)  \quad\text{for all } 
i \text{ and } x\,,
$$ 
then 
any 
supersolution 
 of $\;\sup_\beta \inf_\al L^{\al, \beta}u=0\;$ attaining a minimum in $\Om$ is constant;
 
 \noindent ii)  
 if there exists $\bar \al$ such that
$$
A^{\bar\al,\beta}(x)\geq Z_{i}(x)\otimes Z_{i}(x) \quad\text{for all }  \beta , \,i, \text{ and } x\,,
$$
then any 
supersolution 
 of $\; \inf_\al \sup_\beta L^{\al, \beta}u=0\;$ attaining a minimum in $\Om$ is constant.
\end{cor}
\begin{ex}
If $\mathcal{X}=(X_1,...,X_m)$ are $C^{1,1}$ vector fields on $\overline{\Omega}$ satisfying (BTC), $a, b\in C(\overline\Omega)$ are nonnegative, and $\mathcal{M}^{+}, \mathcal{M}^{-}$ denote the Pucci's extremal operators, then the equation
$$
a(x) \mathcal{M}^{+}((D^2_{\mathcal{X}}u)^*) + b(x)\mathcal{M}^{-}((D^2_{\mathcal{X}}u)^*)=0
$$
is of H-J-I form and satisfies both the SMP and the Strong Minimum Principle.
\end{ex}
\subsection{Other examples and remarks} 
\label{sec:lower}
All the examples of the previous sections satisfy the following property, stronger than Definition \ref{subunit},
\begin{equation}
\label{strong-s}
\lim_{\gamma \to +\infty}F(x,0,p,I-\gamma p\otimes p) = +\infty \quad \forall p\in \R^{d} \;\text{ such that } \; Z\cdot p\ne 0 .
\end{equation}
If $F$ has a SV $Z$ at $x$ satisfying \eqref{strong-s}, then  clearly $Z$ is a SV at $x$ also for any perturbation of $F$ with first or 
zero-th order terms
\[
\bar F(x,r,p,X) = F(x,r,p,X) + H(x,r,p) .
\]
As a consequence, if $F$ satisfies a SMP and $H$ is lower semicontinuous, non-decreasing in $r$, and satisfies (ii) with the same $\phi$ as $F$, then $\bar F$ satisfies the same SMP as $F$.

\begin{ex}
Consider the following perturbation of a Pucci's subelliptic equation associated to H\"ormander vector fields $\mathcal{X}$
\begin{equation*}
c(x)|u|^{k-1}u-a(x)\mathcal{M}^{+}((D^2_{\mathcal{X}}u)^*) + H(x,Du)=0 \,,
\end{equation*}
where $c, a, H$ are continuous and satisfy
\[
c\geq0\ ,\, a>0\ ,
\; 
\text{ either }c=0\text{ or } 1 \leq k\,, 
\; H(x,\xi p)=\xi H(x,p) \;\forall \, \xi>0 \, .
\]
Then the SMP and the Strong Minimum Principle hold, and the same is true if $\mathcal{M}^{+}$ is replaced by  $\mathcal{M}^{-}$.
\end{ex}

Next we give an example of operator 
 that satisfies SMP but whose SV do not satisfy  the stronger property \eqref{strong-s}.
 \begin{ex}
 Consider the equation
\begin{equation}
\label{counter}
 \frac{- \Delta u}{1+|\Delta u|} + f(x) = 0 .
 \end{equation}
 It is easy to see that $F(x,X)= - \Tr X/(1+|\Tr X|)+f(x)$ satisfies condition (i), and also the scaling condition (ii) if $f(x)\geq 0$, by taking $\phi(\xi)=1$ if $\Tr X\geq 0$ and $\phi(\xi)=\xi$ if $\Tr X < 0$. Moreover 
 \begin{equation*}
\lim_{\gamma \to +\infty}F(x,0,p,I-\gamma p\otimes p) = 1+f(x) \quad \forall p\in \R^{d} \,,
\end{equation*}
so any vector $Z\in\R^d$ is SV for $F$ at $x$ if $f(x)>-1$. Then for $f\geq 0$ the equation satisfies the SMP by Remark \ref{gen scal}. However the  stronger property \eqref{strong-s} is not verified for any $Z\in\R^d$.
 \end{ex}
  \begin{ex}{(A counterexample from \cite{KK})} Consider equation \eqref{counter} with $f(x)=0$ for all $x\ne 0$ and $f(0)=-1$. Then (i) holds everywhere, whereas (ii) and the existence of SVs fail only at $x=0$. The SMP is violated by the subsolution $u(x)=0$ for all $x\ne 0$ and $u(0)=1$.
 
     \end{ex}
\section{Strong Comparison Principles}
\label{sec:scp}
In this section we consider non-homogeneous equations that can be written in H-J-Bellman form, namely
 \begin{equation}
 \label{hjb_inf}
\inf_\al \{ L^{\al} u - f^\al (x) \} = 0 \quad \text{ in } \Omega
\end{equation}
 \begin{equation}
\label{hjb_sup}
\sup_\al \{ L^{\al} u - f^\al (x) \} = 0 \quad \text{ in } \Omega
\end{equation}
where $L^\al$ are the linear operators defined in \eqref{linear}. We recall that $F_i$ and $F_s$ defined in \eqref{hjb-hom} are the 1-homogeneous operators obtained by setting $f^\al =0$ in the operator of the equation \eqref{hjb_inf} and  \eqref{hjb_sup}, respectively.
We say that a PDE satisfies the {\em Comparison Principle in a ball} $B(x,r)$ if for any subsolution $u$ and supersolution $v$ in $B(x,r)$ such that $u\leq v$ on $\partial B(x,r)$ we have $u\leq v$ on $\overline B(x,r)$. We will denote
\[
F(x,r,p,X):= \inf_\al \{ -\Tr(A^\al(x) X) - b^\al(x)\cdot p + c^\al(x) r - f^\al(x) \}
\]
\begin{lemma}
Let  $u\in \USC(
\Om)$,  $v\in \LSC(
\Om)$ be, respectively, a sub- and a supersolution of  \eqref{hjb_inf}. 
Assume that for some $\bar r$ the equation  \eqref{hjb_inf} satisfies the Comparison Principle in $B(x,r)$ for all $0<r<\bar r$, and that $F_i$ is continuous and verifies the SMP. If  $u-v$ attains a nonnegative maximum in $\Om$, then $u \equiv v + $constant.
\end{lemma}
\begin{proof}
We claim that $w=u-v$ is a subsolution of $F_i(x,w,Dw,D^2w)=0$. This is easily seen if $u, v$ are smooth because
\[
\inf_\al \{ L^{\al} (u - v)\} \leq \inf_\al \{ L^{\al} u - f^\al (x) - \inf_{\al'} [L^{\al'} u - f^{\al'} (x)] \} \leq 0 \,.
\]
However,  handling the viscosity subsolution property requires more care and the use of the local Comparison Principle. Once the claim is proved the conclusion of the lemma is immediately achieved by the SMP for $F_i$.

We use the compact notations $F[z], F_i[z]$ to denote, respectively, $F(x,z,Dz,D^2z)$ and $F_i(x,z,Dz,D^2z).$ Let $\bar{x}\in\Omega$ and $\varphi$ be a smooth function such that $(w-\varphi)(\bar{x})=0$ and $w-\varphi$ has a strict maximum at $\bar{x}$. Let us argue by contradiction, assuming that $F_i[\varphi(\bar{x})]>0$. We first 
observe that, by the continuity of $F_i$, there exists $\delta>0$ such that
\[
F_i(\bar{x},\varphi(\bar{x})-\delta,D\varphi(\bar{x}),D^2\varphi(\bar{x}))>0\ .
\]
Therefore, 
using the continuity of $F_i$ and the smoothness of $\varphi$, 
we get the existence of $r$ 
 such that
\[
F_i[\varphi-\delta]>0\text{ in }B(\bar{x},r)\ .
\]
Since $w-\varphi$ attains a strict maximum at $\bar{x}$, there exists $0<\eta<\delta$ such that $w-\varphi\leq-\eta<0$ on $\partial B(\bar{x},r)$. We now claim that $v+\varphi-\eta$ satisfies $F[v+\varphi-\eta]\geq0$ in $B(\bar{x},r)$. To this aim, take $\tilde{x}\in B(\bar{x},r)$ and $\psi$ smooth such that $v+\varphi-\eta-\psi$ has a minimum at $\tilde{x}$. Using that $v$ is a 
supersolution of \eqref{hjb_inf},  
denoting by $\tilde{L}^{\alpha}u:=-\Tr(A^{\al}(x) D^2u) - b^{\al}(x)\cdot Du$, we obtain
\begin{multline*}
0\leq F[\psi(\tilde{x})-\varphi(\tilde{x})+\eta]=\inf_{\alpha}\{\tilde{L}^{\alpha}\psi(\tilde{x})-\tilde{L}^{\alpha}\varphi(\tilde{x})+c^\al(\tilde{x})
(\psi(\tilde{x})-\varphi(\tilde{x})+\eta) -f^\al(\tilde{x})\}\\
\leq \inf_{\alpha}\{\tilde{L}^{\alpha}\psi(\tilde{x})+c^\al(\tilde{x})
\psi(\tilde{x})-f^\al(\tilde{x})\}-\inf_{\alpha}\{\tilde{L}^{\alpha}\varphi(\tilde{x})+c^\al(\tilde{x})
(\varphi(\tilde{x})-\eta)\}\\
=F[\psi(\tilde{x})]-F_i[\varphi(\tilde{x})-\eta
]<F[\psi(\tilde{x})]\ .
\end{multline*}
This 
proves the claim that $v+\varphi-\eta$ is a supersolution of \eqref{hjb_inf} in $B(\bar{x},r)$. Now, since $u\leq v+\varphi-\eta$ on $\partial B(\bar{x},r)$, the (local) Comparison Principle yields $u\leq v+\varphi-\eta$ in $B(\bar{x},r)$, in contradiction with the fact that $u(\bar{x})=v(\bar{x})+\varphi(\bar{x})$.
\end{proof}

Now we  can prove the second main result of the paper. We will make the following standard assumptions on the coefficients of $F$:
 \begin{equation}
\label{reg-coeff0}
A^\al(x)=\sigma^\al(x) (\sigma^\al(x))^T \,,\quad \sigma^\al : \overline\Omega \to \{ d\times m \text{ matrices } \}
\end{equation}
 \begin{equation}
\label{reg-coeff1}
\sigma^\al \text{ and } b^\al :  \overline\Omega \to \R^d \;\text{ locally Lipschitz in $x$ 
 uniformly in } \al \,;
\end{equation}
 \begin{equation}
\label{reg-coeff2}
c^\al \geq 0 \,, \quad c^\al \text{ and } f^\al \text{ continuous in $x\in\overline\Omega$ uniformly in } \al \,.
\end{equation}
\begin{thm}
\label{SCP}
Assume \eqref{reg-coeff0}, \eqref{reg-coeff1}, \eqref{reg-coeff2}, and the existence of 
vector fields $Z_{i} : \overline\Om\to \R^d$, $i=1,...,m$, satisfying the H\"ormander condition {\upshape (H)} and such that 
$$
A^\al (x)\geq Z_{i}(x)\otimes Z_{i}(x) \quad\text{for all }  \al , \,i, \text{ and } x\, .
$$ 
 If $u\in \USC(
\Om)$,  $v\in \LSC(
\Om)$ are, respectively, a sub- and a supersolution of  \eqref{hjb_inf} and  $u-v$ attains a nonnegative maximum in $\Om$, then $u \equiv v + $constant.
\end{thm}
\begin{proof}
Under the current assumptions $F$ is finite and continuous in $\overline{\Om}\times\R^d\times\R^d\times\Sym_d$ and it is proper. The homogeneous operator $F_i$ satisfies the SMP by Corollary \ref{smp-hjb}. 

Note that F satisfies the Lipschitz property in $p$ in any compact subset $K\subset\Omega$:
 \begin{equation}
\label{Lip-p}
|F(x,r,p,X)-F(x,r,q,X)| \leq L_K|p-q| \,, \quad \forall\, x\in K .
\end{equation}
Moreover there is $\eta\in C(\overline\Omega)$, $\eta>0$, such that
 \begin{equation}
\label{ell.ty}
F(x,r,p,X+ sI) \leq F(x,r,p,X) - \eta(x) s , \quad \forall\, s>0 .
\end{equation}
In fact, $\Tr (A^\al(x) I)\geq \Tr (Z_{i}(x)\otimes Z_{i}(x)) = |Z_i(x)|^2$ for all $i$, and so
\[
\eta(x) := \frac 1m \sum_{i=1}^m |Z_i(x)|^2
\]
does the job, because the H\"ormander condition prevents that all $Z_i$ vanish at the same point.

By standard viscosity theory \cite{CIL} the equation  \eqref{hjb_inf} verifies the Comparison Principle between a supersolution $v$ and a strict subsolution, say $u_\epsilon$, in a ball $B(x,\bar{r})\subseteq\Omega$ for some $\bar{r}>0$. 
More precisely, $u_\epsilon$ is an upper semicontinuous function in $\overline{B}(x,\bar{r})$ such that
\[
F(x,u_\epsilon,Du_\epsilon,D^2u_{\epsilon})\leq \alpha(x)\text{ in }B(x,\bar{r})
\]
with $\alpha\in C(\overline{B}(x,\bar{r}))$ and $\alpha<0$. If, in addition, $u_\epsilon\rightarrow u$ for all $x$ as $\epsilon$ approaches to 0, then one immediately concludes $u\leq v$ in $B(x,\bar{r})$. 
Next we show that the Comparison Principle holds in all sufficiently small balls, following an argument in \cite{BM}. 
 To this aim, fix $\bar{x}\in \Omega$, $r_1>0$ such that $\overline B(\bar{x}, r_1
 )\subseteq\Omega$,  and let $\bar{\eta}:=\min_{\overline B(\bar{x},r_1
 )}\eta>0$. 
We choose $0<\delta<\bar{\eta}$ and 
$$
\bar{r}:= \min\left(\frac{\bar{\eta}-\delta}{L_K}, r_1\right) ,\quad K:= \overline B(\bar{x}, r_1) .
$$
Consider the function
\[
u_\epsilon(x)=u(x)+\epsilon(e^{\frac{|x-\bar{x}|^2}{2}}-\lambda),\ x\in B(\bar{x},\bar{r})\ .
\]
We claim that $u_\epsilon$ is a strict subsolution in $B(\bar{x},\bar{r})$ for $\lambda$ sufficiently large 
 independent 
 of $\epsilon$. Let us take $\lambda\geq e^{\frac{|x-\bar{x}|^2}{2}}$ for every $x\in B(\bar{x},\bar{r})$ so that $u_\epsilon\leq u$. Straightforward computations yield
\[
(u_\epsilon)_{x_i}=u_{x_i}+\epsilon(x_i-\bar{x})e^{\frac{|x-\bar{x}|^2}{2}}
\]
and 
\[
(u_\epsilon)_{x_ix_j}=u_{x_ix_j}+\epsilon(\delta_{ij}+(x_i-\bar{x})(x_j-\bar{x}))e^{\frac{|x-\bar{x}|^2}{2}}
\]
so that
\[
D^2u_\epsilon=D^2u+\epsilon(I+(x-\bar{x})\otimes(x-\bar{x}))e^{\frac{|x-\bar{x}|^2}{2}}\geq D^2u+\epsilon e^{\frac{|x-\bar{x}|^2}{2}}I
\]
Since $F$ is proper and $u_\epsilon\leq u$, one obtains
\[
F(x,u_\epsilon,Du_\epsilon,D^2u_{\epsilon})\leq F(x,u,Du+\epsilon(x-\bar{x})e^{\frac{|x-\bar{x}|^2}{2}},D^2u+\epsilon e^{\frac{|x-\bar{x}|^2}{2}}I)
\]
Combining \eqref{Lip-p} and \eqref{ell.ty}, one immediately gets
\begin{multline*}
F(x,u,Du+\epsilon(x-\bar{x})e^{\frac{|x-\bar{x}|^2}{2}},D^2u+\epsilon e^{\frac{|x-\bar{x}|^2}{2}}I)\leq F(x,u,Du,D^2u)\\
+\epsilon e^{\frac{|x-\bar{x}|^2}{2}}(L_K|x-\bar{x}|-\eta(x))
\end{multline*}
Using that $u$ is a subsolution and $x\in B(\bar{x},\bar{r})$, by the above choice of $\bar{r}$ we conclude
\[
F(x,u_\epsilon,Du_\epsilon,D^2u_{\epsilon})\leq -\epsilon e^{\frac{|x-\bar{x}|^2}{2}}\delta=:\alpha(x)\ ,
\]
as desired.
\end{proof}

\begin{cor}
\label{WCP}
Under the assumptions of Theorem \ref{SCP} and for bounded $\Omega$, if $u\in \USC(\overline
\Om)$ and  $v\in \LSC(\overline
\Om)$ are, respectively, a sub- and a supersolution of  \eqref{hjb_inf} such that $u\leq v$ in $\partial\Om$, then $u\leq v$ in $\overline\Om$. Moreover, if $u(x)=v(x)$ for some $x\in\Omega$ then $u \equiv v$.
\end{cor}
\begin{proof}
If $\max_{\bar\Omega} (u-v)$ is negative or attained on $\partial \Omega$ the first conclusion is achieved. Otherwise we can apply Theorem \ref{SCP} and get $u(x)-v(x)=k$ for all $x\in \Omega$. Then, for $y\in\partial\Omega$,
\[
k\leq \limsup_{x\to y} (u(x)-v(x))\leq u(y)-v(y)\leq 0 ,
\]
which gives $u\leq v$. Then the last statement follows from Theorem \ref{SCP}.
\end{proof}
\begin{rem} 
The last two results hold also for the equation \eqref{hjb_sup} with convex instead of concave operator. In fact $z=v-u$ is a supersolution of $F_s(x,z,Dz,D^2z)=0$ and we apply the Strong Minimum Principle of Corollary \ref{smp-hjb} {\it ii)} to this equation.
\end{rem} 
\begin{ex}
Theorem \ref{SCP} and Corollary \ref{WCP} apply to the quasilinear equations 
\begin{equation*}
-
\Tr (A(x)(D_{\mathcal{X}}^2u)^*)+ H(x,u,Du)= 0 ,
\end{equation*}
where either $ H=H_i$ or $ H=H_s$ with
\begin{equation*}
H_i(x,r,p):=\inf_\al \{ - b^\al(x)\cdot p + c^\al(x) r - f^\al(x) \}  ,
\end{equation*}
\begin{equation*}
H_s(x,r,p):=\sup_\al \{ - b^\al(x)\cdot p + c^\al(x) r - f^\al(x) \}  ,
\end{equation*}
 the vector fields are $\mathcal{X}=(Z_1,...,Z_m)$,
and the coefficients $A, b^\al, c^\al, f^\al$ satisfy \eqref{reg-coeff1} and \eqref{reg-coeff2}. Also the weak Comparison principle, i.e., the first statement of Corollary \ref{WCP}, is new for these equations, since the results of 
\cite{BM} cover 
either  the case of a Hamiltonian $H$ depending only on the horizontal gradient $D_{\mathcal{X}}u$, or the case where the Lipschitz constant of $H$ w.r.t. $p$ and the diameter of $\overline\Omega$ are  small compared to $\min_{\overline\Omega}\sum_i |Z_i|^2/m$
 (however, in \cite{BM} 
 $H$ 
 is not necessarily concave or convex in $p$).
\end{ex}
\begin{ex}
All the statements of the previous example hold word by word also for the fully nonlinear equations 
\begin{equation*}
\mathcal{M}^{-
}((D_{\mathcal{X}}^2u)^*) + H_i(x,u, Du) 
= 0 ,
\end{equation*}
\begin{equation*}
\mathcal{M}^{+
}((D_{\mathcal{X}}^2u)^*) + H_s(x,u, Du) 
 = 0 .
\end{equation*}
\end{ex}

%
%
%
%
%
%
\end{document}